\documentclass{article}
\usepackage{amsmath,amssymb,amsfonts,amsthm,amscd}
\usepackage{hyperref}
\usepackage{tikz}

\theoremstyle{plain}
\newtheorem{theorem}{Theorem}
\newtheorem{lemma}[theorem]{Lemma}
\newtheorem{proposition}[theorem]{Proposition}

\theoremstyle{definition}
\newtheorem{definition}[theorem]{Definition}

\newtheorem{example}[theorem]{Example}

\theoremstyle{remark}

\numberwithin{theorem}{section}

\title{Symmetric Graphs have symmetric Matchings}
\author{Jan Fricke}

\begin{document}

\maketitle

\begin{abstract}
  Assume that there is a free group action of automorphisms on a
  bipartite graph. If there is a perfect matching on the factor graph,
  then obviously there is a perfect matching on the
  graph. Surprisingly, the reversed is also true for amenable groups:
  if there is a perfect matching on the graph, there is
  also a perfect matching on the factor graph, i. e. a group invariant
  (``symmetric'') perfect matching on the graph.
\end{abstract}

\section{Introduction}
In 2007 we had an interesting discussion about a problem that somehow
points to Erd\H{o}s. Take $\mathbb{Z}^2$ in $\mathbb{R}^2$ and rotate
that by $45^\circ$. Is there a bijection $f$ between the points of the
original lattice and the rotated lattice, such that the distance
between $x$ and $f(x)$ is bounded? We could confirm that (and that was
already known), but is spurred us to find the smallest bound. Klaus
Nagel found a nice construction of a bijection, and in a talk in July
2008 he proved that this one has the best bound, which is
$\sqrt{5}\cdot\sin\tfrac\pi8\approx0.8557$. Unfortunately, he never
published that, and I made no notices.

Independent from that I tried to analyze the dependence of that bound
from rotation angle. And, there are in fact two qualitatively
different cases. If the rotation matrix is irrational, the ``twin
lattice'' is aperiodic, and the bound doesn't change if one of the
lattices is translated. If the rotation matrix is rational, the ``twin
lattice'' is periodic, and the bound depends on the rotation
center. Nevertheless there is some continuous dependence of that bound
from the rotation angle, and it is a nice fractal function. Since I
cannot prove all of that it is still unpublished, but I will write a
survey paper soon.

In doing numerical experiments on that problem, I faced the problem of
finding a matching in an infinite but periodic bipartite graph. So I
only considered periodic matchings which could be done by the
well-known Hopcroft-Karp algorithm. But does that always work?
Consider the following graph with a rotation symmetry:
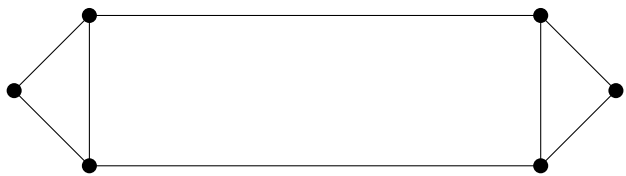
\begin{center}
  \begin{tikzpicture}
    \fill (3,1) circle (0.1);
    \fill (3,-1) circle (0.1);
    \fill (4,0) circle (0.1);
    \fill (-3,1) circle (0.1);
    \fill (-3,-1) circle (0.1);
    \fill (-4,0) circle (0.1);
    \draw (3,1) -- (4,0) -- (3,-1) -- (3,1) -- (-3,1)  -- (-4,0) --
    (-3,-1) -- (-3,1) (-3,-1) -- (3,-1);
  \end{tikzpicture}
\end{center}
It has two different perfect matchings, but none of them has the
rotation symmetry.

So I was thinking about that in 2009 while skiing at the Kvitfjell in
Norway, when I found out, that it works on bipartite graphs using the
Theorem of Hall, because the ``boundary'' of a suitable finite subset
of $\mathbb{Z}^2$ is ``small'' relative to the ``interior''. So I
could prove a general result for groups with subexponential growth. I
suspected that it has something to do with amenability and
unsuccessful tried to find a counterexample for the free group of two
generators. So it remained unpublished.

In April 2016 I wanted to present that proof in a lecture, and
suddenly I found the slight modification of the proof for showing it
for all amenable groups, and the paradoxical decomposition gave me the
right hint for counterexamples for all non-amenable groups. You will
find that proof and the counterexamples in the last section. In the
sections in between I collect the well-known facts about the Theorem
of Hall (section 2), define symmetry of a bipartite graph and the
factor graph (section 3), and cite equivalent characterizations of
amenability and a slight modification of one of them (section 4).

I want to thank Rainer Rosenthal, who still patiently waits for Klaus
Nagel's paper and my survey paper.

\section{Matchings}

A \emph{bipartite graph} $(A,B,E)$ is a pair of sets $A$ and $B$ with
a subset $E\subseteq A\times B$. The elements of $A$ and $B$ are
\emph{vertices}, the elements of $E$ are \emph{edges}. For each subset
$X\subseteq A$ we define $E(X)=\{y\in B\mid (x,y)\in E\}$ and for each
$Y\subseteq B$ we define $E(Y)=\{x\in A\mid (x,y)\in E\}$. Such a
bipartite graph is called \emph{locally finite}, if for any finite
$X\subseteq A$ also $E(X)$ is finite, and for any finite $Y\subseteq
B$ also $E(Y)$ is finite. (Equivalently, for any $x\in A$ and $y\in B$
the sets $E(\{x\})$ and $E(\{y\})$ are finite.)

A bipartite graph $(A,B,E)$ fulfills the \emph{left Hall condition},
if $|X|\le |E(X)|$ for any finite $X\subseteq A$, and the \emph{right
  Hall condition}, if $|Y|\le |E(Y)|$ for any finite $Y\subseteq
B$. It fulfills the \emph{Hall condition}, if the left and right Hall
condition are fulfilled.

A \emph{matching} $M$ of a bipartite graph is a subset $M\subseteq E$,
such that for all $(x_1,y_1),(x_2,y_2)\in M$ there is
$(x_1,y_1)=(x_2,y_2)$ whenever $x_1=x_2$ or $y_1=y_2$. Define
$A(M)=\{x\in A\mid \exists y\in B:(x,y)\in M\}$ and
$B(M)=\{y\in B\mid \exists x\in A:(x,y)\in M\}$. A matching $M$ is
called \emph{perfect}, if $A(M)=A$ and $B(M)=B$.

\begin{theorem}[Hall]
  For a locally finite bipartite graph there are equivalent:
  \begin{enumerate}
  \item The graph fulfills the Hall condition.
  \item The graph has a perfect matching.
  \end{enumerate}
\end{theorem}

\section{Symmetry}

\begin{definition}
  Let $G$ be a group and $(A,B,E)$ a bipartite graph. Then the graph
  is called \emph{$G$-symmetric}, if
  \begin{itemize}
  \item $G$ acts free on $A$,
  \item $G$ acts free on $B$,
  \item $(x,y)\in E \iff (gx,gy)\in E$ for all $x\in A$, $y\in B$ and
    $g\in G$.
  \end{itemize}
  It is called \emph{proper $G$-symmetric}, if over-more
  \begin{itemize}
  \item $(x,y)\in E$ and $(gx,y)\in E$ implies $g$ is the
    identity, or equivalently
  \item $(x,y)\in E$ and $(x,gy)\in E$ implies $g$ is the identity.
  \end{itemize}
\end{definition}

\begin{definition}
  A matching $M$ on a $G$-symmetric bipartite graph is called
  \emph{$G$-symmetric}, if $(x,y)\in M$ implies $(gx,gy)\in M$ for any
  $g\in G$.
\end{definition}

A $G$-symmetric matching can also be described by a matching on some
factor graph:

\begin{definition}
  Let $(A,B,E)$ a $G$-symmetric bipartite graph. Then the \emph{factor
    graph} $(\tilde{A},\tilde{B},\tilde{E})$ is defined by:
  \begin{itemize}
  \item $\tilde{A}=A/G=\{Gx\mid x\in A\}$,
  \item $\tilde{B}=B/G=\{Gy\mid y\in B\}$,
  \item $\tilde{E}=\{(Gx,Gy)\mid (x,y)\in E\}$.
  \end{itemize}
\end{definition}

\begin{lemma}
  If $(A,B,E)$ is locally finite, then also 
  $(\tilde{A},\tilde{B},\tilde{E})$ is locally finite.
\end{lemma}

\begin{proof}
  Let $(A,B,E)$ be locally finite. Then
  \begin{align*}
    \tilde{E}(\{Gx\})&=\{Gy\mid (Gx,Gy)\in\tilde{E}\}
    =\{Gy\mid (gx,y)\in E \text{ and } g\in G\} \\
    &=\{Gy\mid (x,g^{-1}y)\in E \text{ and } g\in G\}
    =\{Gy'\mid (x,y')\in E\},
  \end{align*}
  hence finite. The same holds for $\tilde{E}(\{Gy\})$.
\end{proof}


\begin{proposition}
  The map $M\mapsto\tilde{M}=\{(Gx,Gy)\mid (x,y)\in M\}$ is a
  surjection between the $G$-symmetric matchings $M$ of $(A,B,E)$ and
  the matchings $\tilde{M}$ of $(\tilde{A},\tilde{B},\tilde{E})$.
  Is $(A,B,E)$ even proper $G$-symmetric this map is a bijection.
  Over-more, $M$ is perfect if and only if $\tilde{M}$ is perfect.
\end{proposition}

\section{Amenable groups}
There are some equivalent definitions of amenable groups using means
on groups or invariant probability measures. Since that properties
doesn't matter here, we just choose that one:

\begin{definition}
  A group $G$ is called \emph{amenable}, if there
  is an invariant finitely additive probability measure on $G$.
\end{definition}

The property ``amenable'' is equivalent to some other nice properties
of the group. In particular, we need the F\o{}lner condition and the
paradoxical decomposition. 

\begin{definition}
  For a group $G$ note by $\mathcal{F}$ the family of all non-empty
  finite subsets.

  A group $G$ satisfies the \emph{F\o{}lner condition}, if for any
  $\varepsilon>0$ and any $U\in\mathcal{F}$ there is a
  $F\in\mathcal{F}$ with $|F\setminus Fg|\le
  \varepsilon\cdot |F|$ for any $g\in U$.

  A group $G$ has a \emph{paradoxical decomposition}, if there is a
  $F\in\mathcal{F}$ and a decomposition 
  \begin{equation*}
    G=\bigsqcup_{g\in F}A_g=\bigsqcup_{g\in F}B_g=
    \bigsqcup_{g\in F}A_gg \sqcup \bigsqcup_{g\in F}B_gg.
\end{equation*}
\end{definition}

\begin{theorem}[F\o{}lner, Tarski]
  For a group $G$ are equivalent:
  \begin{enumerate}
  \item $G$ is amenable.
  \item $G$ satisfies the F\o{}lner condition.
  \item $G$ has no paradoxical decomposition.
  \end{enumerate}
\end{theorem}

Both of that properties are interesting here, but we need a slight
modification of the F\o{}lner condition:

\begin{proposition}
  A group $G$ satisfies the F\o{}lner condition if and only if
  \begin{equation*}
    \inf_{F\in\mathcal{F}} \frac{|FU|}{|F|} = 1
  \end{equation*}
  for any $U\in\mathcal{F}$, where $\mathcal{F}$ is the family of all
  non-empty finite subsets of $G$.
\end{proposition}

\begin{proof}
  Note that $|F|\le|FU|$ for all $F,U\in\mathcal{F}$, so
  $\inf\limits_{F\in\mathcal{F}} \frac{|FU|}{|F|} \ge 1$.

  ``$\Rightarrow$'': Let $U\in\mathcal{F}$ and $\varepsilon>0$. By the
  F\o{}lner condition there is a set $F\in\mathcal{F}$ such that
  \begin{equation*}
    |F\setminus Fg|\le \frac{\varepsilon}{|U|} \cdot |F|
  \end{equation*}
  for any $g\in U$. Since $|F|=|Fg|$ we have $|F\setminus
  Fg|=|Fg\setminus F|$, hence
  \begin{equation*}
    |FU\setminus F|=
    \left|\left(\bigcup_{g\in U}Fg\right)\setminus F\right|=
    \left|\bigcup_{g\in U}(Fg\setminus F)\right|
    \le \varepsilon\cdot |F|.
  \end{equation*}
  Now
  \begin{equation*}
    |FU| = |(FU\setminus F)\cup (FU \cap F)|
    = |FU\setminus F| + |FU \cap F|
    \le \varepsilon\cdot |F| + |F|
  \end{equation*}
  implies the assertion.

  ``$\Leftarrow$'': Let $U\in\mathcal{F}$ and $\varepsilon>0$. Then
  there is a set $F\in\mathcal{F}$ such that
  \begin{equation*}
    |F(U\cup\{e\})|/|F|<1+\varepsilon.
  \end{equation*}
  For any $g\in U$ we have
  \begin{equation*}
    |F(U\cup\{e\})| = |FU \cup F| = |(FU\setminus F) \cup F|
    =|FU\setminus F|+|F|\ge|Fg\setminus F|+|F|.
  \end{equation*}
  This implies $|F\setminus Fg|=|Fg\setminus F|<\varepsilon\cdot|F|$.
\end{proof}

\begin{example}
  Finite groups are amenable.

  Solvable groups are amenable, in particular abelian and nilpotent
  groups are amenable.

  Any subgroup of an amenable group is amenable.

  If $U$ is an amenable normal subgroup of $G$, and $G/U$ is amenable
  too, then $G$ is amenable.

  Any group with subexponential growth is amenable.
\end{example}

\begin{example}
  Any free group on at least two generators is not amenable.

  Any group that has a non-amenable subgroup is not amenable,
  e.g. $\operatorname{SL}(n,\mathbb{Z})$ for $n\ge 2$, since it has a
  subgroup isomorphic to the free group on two generators.
\end{example}

\section{Symmetric graphs have symmetric matchings}
Now we can formulate and proof the main result:

\begin{theorem}
  Let $(A,B,E)$ a locally finite $G$-symmetric bipartite graph, where
  $G$ is amenable. Then the following properties are
  equivalent.
  \begin{enumerate}
  \item $(A,B,E)$ has a perfect matching.
  \item $(A,B,E)$ has a perfect $G$-symmetric matching.
  \end{enumerate}
\end{theorem}

Since the reverse direction is obvious, there is only one direction to
show.
\begin{proof}
  Let $\tilde{X}\subseteq\tilde{A}$ be some finite subset in the
  factor graph. Set $\tilde{Y}=\tilde{E}(\tilde{X})$. Then there are
  representing sets $X\subseteq A$ and $Y\subseteq B$, i.e.
  \begin{align*}
    \tilde{X} &= \{Gx\mid x\in X\}, &
    \tilde{Y} &= \{Gy\mid y\in Y\}, &
    |X| &= |\tilde{X}|, &
    |Y| &= |\tilde{Y}|.
  \end{align*}
  Now for any pair $(x,y)\in E(X)$ there is some $g\in G$ with
  $gy\in Y$, so there is a $U\in\mathcal{F}$ with
  $E(X)\subseteq UY$. Since $E$ is $G$-invariant, we conclude
  $E(FX)\subseteq FUY$ for any $F\in\mathcal{F}$.

  By the left Hall condition, the free action of $G$, and the modified
  F\o{}lner condition this implies
  \begin{equation*}
    |F|\cdot|X|\le |FU|\cdot |Y|
    \implies
    |X|\le \frac{|FU|}{|F|}\cdot |Y|
    \implies
    |X| \le |Y|
    \implies
    |\tilde{X}| \le |\tilde{Y}|    
  \end{equation*}
  which proves the left Hall condition in the factor graph. The right
  Hall condition is shown the same way, so the assertion is proved.
\end{proof}

\begin{proposition}
  If $G$ is not amenable, then there is a locally finite proper
  $G$-symmetric bipartite graph $(A,B,E)$ with a perfect matching,
  that admits no $G$-symmetric perfect matching.
\end{proposition}

\begin{proof}
  By the Tarski Theorem there is a paradoxical decomposition
  of $G$. This can be used to construct such a $G$-symmetric bipartite
  graph: Set $A=G$, $B=G\times\{1,2\}$, $E=\{(xg,(x,i))\mid x\in G,
  g\in F,i\in\{1,2\}\}$, and the group action is the group
  composition. Then there is a perfect matching
  \begin{equation*}
    M = \{ (xg,(x,1)) \mid g\in F, x\in A_g\}
    \cup  \{ (xg,(x,2)) \mid g\in F, x\in B_g\}.
  \end{equation*}
  But the factor graph $\tilde{A}=\{G\}$,
  $\tilde{B}=\{G\}\times\{1,2\}$, $\tilde{E}=\tilde{A}\times\tilde{B}$
  has obviously no perfect matching.

  It's easy to see that this example is not proper $G$-symmetric, so
  we have to modify it slightly. Each vertex has to be replaced by
  $|F|$ copies of itself, and the edges have to be twisted a little bit.

  Let $\varphi:F\times F\to F$ be
  bijective if any of the arguments is fixed. One can get such a
  function for instance by the group composition in a cyclic group of
  order $|F|$, and map the elements bijectively to $F$. Now we set
  $A = G\times F$, $B = G\times F\times\{1,2\}$, and 
  \begin{equation*}
    E = \{((xg,\varphi(g,h)),(x,h,i))\mid
    g,h\in F, x\in G, i\in \{1,2\}\}.
  \end{equation*}
  $G$ acts again by left multiplication on the $G$-component of $A$
  and $B$. This $G$-symmetry is proper, because
  \begin{equation*}
    ((zxg,\varphi(g,h)),(x,h,i)) = ((x'g',\varphi(g',h')),(x',h',i'))
  \end{equation*}
  for some $z\in G$ implies $x=x'$, $h=h'$, $i=i'$, and by injectivity
  of $g\mapsto\varphi(g,h)$ also $g=g'$, hence $z=e$.

  A perfect matching is given by
  \begin{align*}
    M =& \{((xg,\varphi(g,h)),(x,h,1))\mid g,h\in F, x\in A_g\}
    \cup \\
    & \{((xg,\varphi(g,h)),(x,h,2))\mid g,h\in F, x\in B_g\} \subseteq
    E.
  \end{align*}
  Since $G=\bigsqcup\limits_{g\in F}A_g=\bigsqcup\limits_{g\in F}B_g$
  any vertex of $B$, and since
  $G=\bigsqcup\limits_{g\in F}A_gg \sqcup \bigsqcup\limits_{g\in F}B_gg$
  and $h\mapsto\varphi(g,h)$ is bijective, any vertex of $A$
  belongs to exactly one edge in $M$. Hence, $(A,B,E)$ admits a
  perfect matching.

  Now take a look at the factor graph. We have
  $\tilde{A}=\{G\}\times F$, $\tilde{B}=\{G\} \times F\times\{1,2\}$,
  \begin{align*}
    \tilde{E} &=  \{((G,\varphi(g,h)),(G,h,i))\mid
    g,h\in F, i\in \{1,2\}\},
  \end{align*}
  which is the complete bipartite graph with $|F|$ and $2|F|$
  vertices. That one has obviously no perfect matching.
\end{proof}

\end{document}